\documentclass[10pt]{amsart}
\usepackage{amssymb,mathrsfs,graphicx,subfigure, enumerate}
\usepackage{amsmath,amsfonts,amssymb,amscd,amsthm,bbm}
\usepackage{extpfeil}
\usepackage{graphicx,colortbl}
\usepackage{float}
\usepackage{epsfig}
\usepackage{caption}
\usepackage{bm}
\graphicspath{{Figure/}}

\usepackage[margin=1in]{geometry}

\title[Propagation of the mono-kinetic solution in the collective dynamics]{Propagation of the mono-kinetic solution in the Cucker-Smale-type kinetic equations}

\author[Kang]{Moon-Jin Kang}
\address[Moon-Jin Kang]{\newline Department of Mathematics  and Research Institute of Natural Sciences
\newline Sookmyung Women's University, Seoul  04310, Korea}
\email{moonjinkang@sookmyung.ac.kr}

\author[Kim]{Jeongho Kim}
\address[Jeongho Kim]{\newline Institute of New Media and Communications \newline Seoul National University, Seoul  08826, Korea}
\email{jhkim206@snu.ac.kr}

\begin{document}
\newtheorem{theorem}{Theorem}[section]
\newtheorem{lemma}{Lemma}[section]
\newtheorem{corollary}{Corollary}[section]
\newtheorem{proposition}{Proposition}[section]
\newtheorem{remark}{Remark}[section]
\newtheorem{definition}{Definition}[section]

\renewcommand{\theequation}{\thesection.\arabic{equation}}
\renewcommand{\thetheorem}{\thesection.\arabic{theorem}}
\renewcommand{\thelemma}{\thesection.\arabic{lemma}}
\newcommand{\bbr}{\mathbb R}
\newcommand{\bbz}{\mathbb Z}
\newcommand{\bbn}{\mathbb N}
\newcommand{\bbs}{\mathbb S}
\newcommand{\bbp}{\mathbb P}
\newcommand{\bbt}{\mathbb T}
\newcommand{\<}{\langle}
\renewcommand{\>}{\rangle}
\newcommand{\T}{\mathbb{T}}
\newcommand{\N}{\mathbb{N}}
\newcommand{\R}{\mathbb{R}}
\newcommand{\lt}{\left}
\newcommand{\rt}{\right}
\newcommand{\bq}{\begin{equation}}
\newcommand{\eq}{\end{equation}}
\newcommand{\beq}{\begin{equation}}
\newcommand{\eeq}{\end{equation}}
\newcommand{\e}{\varepsilon}
\newcommand{\mc}{\mathcal{C}}
\newcommand{\pa}{\partial}



\thanks{\textbf{Acknowledgment.} The work of  M.-J. Kang is partially supported by the NRF-2019R1C1C1009355. The work of J. Kim was supported by the Basic Research Lab Program through the National Research Foundation of Korea (NRF) funded by the MSIT(2018R1A4A1059976).}

\begin{abstract} 
	In this paper, we study the propagation of the mono-kinetic distribution in the Cucker-Smale-type kinetic equations. More precisely, if the initial distribution is a Dirac mass for the variables other than the spatial variable, then we prove that this ``mono-kinetic" structure propagates in time. For that, we first obtain the stability estimate of measure-valued solutions to the kinetic equation, by which we ensure the uniqueness of the mono-kinetic solution in the class of measure-valued solutions with compact supports. We then show that the mono-kinetic distribution is a special measure-valued solution. The uniqueness of the measure-valued solution implies the desired propagation of mono-kinetic structure.
\end{abstract}

\maketitle

\section{Introduction}\label{sec:1}
\setcounter{equation}{0}

The collective dynamics is one of the most interesting phenomena that can be found in the nature and society. The flocking of birds or the flow of pedestrians are the best examples of such phenomena. For decades, the models in the collective dynamics, such as the Vicsek model \cite{V-C-B-C-S} or the Cucker-Smale (in short, C-S) model \cite{CS07} have been studied extensively. These models were started from the microscopic model, which describes the dynamics of the position and velocity of each single particle, interacting with the other particles. Moreover, inspired by the kinetic theory of molecular gases and fluid dynamics, the mesoscopic and macroscopic descriptions for the models were developed \cite{H-Liu,H-T} for describing the dynamics when the number of agents is very large. More precisely, mesoscopic and macroscopic descriptions for the C-S model are respectively presented as follows \cite{H-T}:
\begin{align}
\begin{aligned}\label{eq_kinet_CS}
&\partial_t f+ v\cdot \nabla_x f+ \nabla_v \cdot [F[f]f]=0,\quad (t,x,v)\in \bbr_+\times\bbt^d\times\bbr^d,\\
&F[f](t,x,v):=\int_{\bbt^d\times\bbr^{d}}\phi(x-x_*)(v_*-v)f(t,x_*,v_*)\,dx_*\,dv_*;
\end{aligned}
\end{align}
\begin{align}
\begin{aligned}\label{eq_hydro_CS}
&\partial_t \rho +\nabla_x \cdot (\rho u) =0,\quad (t,x)\in \bbr_+\times\bbt^d,\\
&\partial_t (\rho u) +\nabla_x\cdot(\rho u \otimes u)=\rho\int_{\bbt^d} \phi(x-x_*)(u(x_*)-u(x))\rho(x_*)\,dx_*.
\end{aligned}
\end{align}
The hydrodynamic C-S equations \eqref{eq_hydro_CS} can be formally derived from the kinetic C-S equation \eqref{eq_kinet_CS} by adopting the mono-kinetic ansatz:
\[f(t,x,v)=\rho(t,x) \otimes \delta_{u(t,x)}(v),\]
where $\delta_u(v)$ denotes the Dirac mass concentrated at $u$. For a rigorous derivation from \eqref{eq_kinet_CS} to \eqref{eq_hydro_CS}, we refer to \cite{F-K}, in which the hydrodynamic limit of \eqref{eq_kinet_CS} with a strong local alignment was rigorously proved. \\

On the other hand, the C-S model was generalized to the thermomechanical Cucker-Smale (TCS) model that takes into account the effect of the internal variables, such as temperature \cite{HR17}. The kinetic and hydrodynamic systems for the TCS model are respectively given by
\begin{align}
\begin{aligned}\label{eq_kinet_TCS}
&\partial_t f+ v\cdot \nabla_x f+ \nabla_v \cdot [F[f]f]+\partial_\theta [G[f]f]=0,\quad (t,x,v,\theta)\in \bbr_+\times\bbt^d\times\bbr^d\times\bbr_+,\\
&F[f](t,x,v,\theta)=\int_{\bbt^d\times\bbr^{d}\times\bbr_+}\phi(x-x_*)\left(\frac{v_*}{\theta_*}-\frac{v}{\theta}\right)f(t,x_*,v_*,\theta_*)\,dx_*\,dv_*\,d\theta_*,\\
&G[f](t,x,\theta)=\int_{\bbt^d\times\bbr^{d}\times\bbr_+}\zeta(x-x_*)\left(\frac{1}{\theta}-\frac{1}{\theta_*}\right)f(t,x_*,v_*,\theta_*)\,dx_*\,dv_*\,d\theta_* ,\\
\end{aligned}
\end{align}
subject to the initial data $f(0,x,v,\theta)=f_0(x,v,\theta)$, and 
\begin{align}
\begin{aligned}\label{eq_hydro_TCS}
&\partial_t \rho +\nabla_x \cdot (\rho u) =0,\quad (t,x)\in \bbr_+\times\bbt^d,\\
&\partial_t (\rho u) +\nabla_x\cdot(\rho u \otimes u)=\rho\int_{\bbt^d} \phi(x-x_*)\left(\frac{u(x_*)}{e(x_*)}-\frac{u(x)}{e(x)}\right)\rho(t,x_*)\,dx_*,\\
&\partial_t(\rho e)+\nabla_x\cdot(\rho u e)=\rho \int_{\bbt^{d}}\zeta(x-x_*)\left(\frac{1}{e(x)}-\frac{1}{e(x_*)}\right)\rho(t,x_*)\,dx_* ,
\end{aligned}
\end{align}
subject to the initial data $\rho(0,x):=\rho_0(x)$, $u(0,x):=u_0(x)$ and $e(0,x):=e_0(x)$ respectively.\\

For a rigorous study on the hydrodynamic limit of the kinetic equation \eqref{eq_kinet_TCS}, we refer to \cite{KHKS}. There, they proved a hydrodynamic limit of \eqref{eq_kinet_TCS} with a strong local alignment towards \eqref{eq_hydro_CS}, by considering the temperature support of the initial data $f_0$ degenerating to a single value as the scaling parameter tends to 0. However, the hydrodynamic limit from \eqref{eq_kinet_TCS} toward \eqref{eq_hydro_TCS} is still an open and challenging problem. 
The main difficulties in the limit process from \eqref{eq_kinet_TCS} to \eqref{eq_hydro_TCS} are due to the severe singularity of the mono-kinetic distribution, and the strong nonlinearity of the nonlocal interaction. 
For the other results on these kinds of singular limit leading to the mono-kinetic distribution, we refer to \cite{JR,K18,KV}.\\

However, for a rigorous justification on the mono-kinetic ansatz, it is natural to ask the following question; does the solution $f$ of the kinetic equation with the mono-kinetic initial data $f_0$ preserves the mono-kinetic property? More precisely, if the initial data $f_0$ is given by
\[f_0(x,v,\theta)=\rho_0(x)\otimes\delta_{u_0(x)}(v)\otimes \delta_{e_0(x)}(\theta),\quad (\mbox{resp. }f_0(x,v)=\rho_0(x)\otimes\delta_{u_0(x)}(v) ~\mbox{for the C-S model}),\]
then, does the solution $f$ is also of the mono-kinetic form given by
\[
f(t,x,v,\theta)=\rho(t,x)\otimes\delta_{u(t,x)}(v)\otimes \delta_{e(t,x)}(\theta), \quad (\mbox{resp. } f(t,x,v)=\rho(t,x)\otimes\delta_{u(t,x)}(v)~\mbox{for the C-S model})
\]
for some functions $\rho(t,x)$, $u(t,x)$ and $e(t,x)$? Note that considering the derivation of hydrodynamic equations, $\rho,u$ and $e$ should be given as the solutions of the hydrodynamic system \eqref{eq_hydro_TCS}. \\

In this article, we aim to give a rigorous answer to the above question, by obtaining the stability and uniqueness of the mono-kinetic solution in some class. \\
We will only focus on the TCS models \eqref{eq_kinet_TCS} and \eqref{eq_hydro_TCS} for the above question, because the same result also holds in the simpler case of the C-S model.

\section{Preliminaries and main theorems}\label{sec:2}
\setcounter{equation}{0}
In this section, we provide the basic definitions, previous results and the main theorem of this paper.

\subsection{Preliminaries}
We first provide the definitions of the measure-valued solutions and bounded Lipschitz distances, and we also present the existence and uniqueness of the smooth solution to the hydrodynamic equations. We define $\mathcal{M}(X)$ as a set of  nonnegative Radon measures defined on $X:=\bbt^d\times\bbr^d\times\bbr_+$. For a measure $\mu$ and $g\in C_0(X)$, we define
\[\<\mu,g\>:=\int_{X} g\,\mu(dz),\quad z\in X.\]

\begin{definition}\label{D2.1}\cite{HKMRZ-2}
	The time dependent measure $\mu=\mu_t\in L^\infty([0,T]; \mathcal{M}(X))$ is said to be a measure-valued solution to \eqref{eq_kinet_TCS} with the initial measure $\mu_0\in\mathcal{M}(X)$ if the following conditions hold:
	\begin{enumerate}
		\item $\mu$ is weakly continuous in time: for any $g\in C_0(X)$, the map $t\mapsto \<\mu_t,g\>$ is continuous.
		\item $\mu$ satisfies \eqref{eq_kinet_TCS} in the sense of distribution: for any $g\in C_0^1([0,T)\times X)$,
		\begin{align}
		\begin{aligned}\label{eq_measure}
		\<\mu_t,g(\cdot,t)\>-\<\mu_0,g(\cdot,0)\>&=\int_0^t \<\mu_s,\partial_sg+v\cdot \nabla_x g+F[\mu_s]\cdot \nabla_v g+G[\mu_s]\partial_\theta g\>\,ds,
		\end{aligned}
		\end{align}
		where $F[\mu_t](z)$ and $G[\mu_t](x,\theta)$ are defined as
		\begin{align}
		\begin{aligned}\label{FG}
		&F[\mu_t](x,v,\theta):=\int_{X}\phi(x-x_*)\left(\frac{v_*}{\theta_*}-\frac{v}{\theta}\right)\,\mu_t(dz_*),\\
		&G[\mu_t](x,\theta):=\int_{X}\zeta(x-x_*)\left(\frac{1}{\theta}-\frac{1}{\theta_*}\right)\,\mu_t(dz_*).
		\end{aligned}
		\end{align}
	\end{enumerate}
\end{definition}

\vspace{0.3cm}

\noindent We now consider the following subset $\Omega$ of bounded and Lipschitz continuous functions
\[\Omega:=\left\{g:X\to\bbr~|~\|g\|_{L^\infty}\le 1,\quad \|g\|_{\textup{\textup{Lip}}}\le1\right\}.\]
Then, for any two measures $\mu,\nu$ on $X$, we define the bounded Lipschitz distance $d(\mu,\nu)$ as
\[d(\mu,\nu):=\sup_{g\in\Omega}\left|\int_{X}g(z)(\mu-\nu)(dz)\right|.\]
It is well-known that for any bounded and Lipschitz continuous function $g\in C_0(X)$, 
\begin{equation}\label{B-1}
\left|\int_{X}g(z)(\mu-\nu)(dz)\right|\le \max\{\|g\|_{L^\infty},\|g\|_{\textup{Lip}}\}d(\mu,\nu).
\end{equation}

In the following, we present the global well-posedness of the hydrodynamic model \eqref{eq_hydro_TCS}.
\begin{proposition}\cite{HKMRZ-1} \label{prop:smooth}
	 Let $s>\frac{d}{2}+1$. Suppose that 
	  \begin{equation}\label{ini-htcs}
	 (\rho_0,u_0,e_0)\in H^s(\bbt^d)\times H^{s+1}(\bbt^d)\times H^{s+1}(\bbt^d),\quad\mbox{together with some smallness condition}.
	  \end{equation}
	  Then, there exists a unique classical solution $(\rho,u,e)$ to \eqref{eq_hydro_TCS} satisfying
	 \[\rho\in C^0(0,\infty;H^s(\bbt^d))\cap C^1(0,\infty;H^{s-1}(\bbt^d)),\quad u\in C^0(0,\infty;H^{s+1}(\bbt^d))\cap C^1(0,\infty;H^{s}(\bbt^d)),\] 
	\[e\in C^0(0,\infty;H^{s+1}(\bbt^d))\cap C^1(0,\infty;H^{s}(\bbt^d)).\]
\end{proposition}

\subsection{Main theorem}
We are now ready to provide the main theorem of the paper.

\begin{theorem}\label{T2.2}
Assume that the kernels $\phi$ and $\zeta$ are Lipschitz continuous in $\bbt^d$. 
For a given $T>0$, let $\mu,\nu\in L^\infty([0,T]; \mathcal{M}(X))$ be measure valued solutions to \eqref{eq_kinet_TCS} 
with compact supports for each time $t\in[0,T]$.
	Then, there exists $C_T>0$ such that for any $0\le t\le T$, 
	\begin{equation}\label{C-2}
	d(\mu_t,\nu_t)\le C_Td(\mu_0,\nu_0).
	\end{equation}
In particular, consider a mono-kinetic initial datum $f_0(x,v,\theta)=\rho_0(x)\otimes\delta_{u_0(x)}(v)\otimes \delta_{e_0(x)}(\theta)$, where $(\rho_0,u_0,e_0)$ satisfies \eqref{ini-htcs}. 
Then, the kinetic equation \eqref{eq_kinet_TCS} has a unique measure-valued solution
	\[f(t,x,v,\theta)=\rho(t,x)\otimes\delta_{u(t,x)}(v)\otimes \delta_{e(t,x)}(\theta), \quad t\in[0,T], \]
in the class of measure-valued solutions to \eqref{eq_kinet_TCS} with compact supports for all $t\in[0,T]$.
Here, $(\rho,u,e)$ represents the classical solution to the hydrodynamic system \eqref{eq_hydro_TCS} with the initial datum $(\rho_0,u_0,e_0)$. \\
In other words, the mono-kinetic distribution of the solution to \eqref{eq_kinet_TCS} propagates in time.
\end{theorem}

\begin{remark}
The above theorem also holds in the simpler case of the C-S models \eqref{eq_kinet_CS} and \eqref{eq_hydro_CS}. Indeed, the existence of smooth solutions to \eqref{eq_hydro_CS} as in Proposition \ref{prop:smooth} was proved in \cite{HKK}, and the stability estimate of measure-valued solutions to \eqref{eq_kinet_CS} as in  Proposition \ref{prop:st} was proved in \cite[Proposition 5.10]{H-Liu}. Moreover, the computations in Section \ref{sec:4} also work in the case of the C-S model. 
\end{remark}

\section{Stability of measure-valued solutions}\label{sec:3}
\setcounter{equation}{0}
In this section, we present the stability of measure-valued solutions to the kinetic TCS equation \eqref{eq_kinet_TCS} in terms of the bounded Lipschitz distance. The goal of this section is to prove the following proposition.
\begin{proposition}\label{prop:st}
Let $T>0$ and $\mu,\nu\in L^\infty(0,T; \mathcal{M}(X))$ be measure-valued solutions to \eqref{eq_kinet_TCS} with compact supports for each time $t\in[0,T]$, that is, there exist positive constants $P_T, \theta_m^0, \theta_M^0$ such that
	\begin{equation}\label{supp-ass}
	\mbox{supp}(\mu_t),\mbox{supp}(\nu_t)\subset \bbt^d\times B_{P_T}(0)\times\left[\theta_m^0,\theta_M^0\right],\quad\mbox{for any $0\le t\le T$}.
	\end{equation}
	Then, there exists $C_T>0$ such that for any $0\le t\le T$, 
	\begin{equation}\label{C-2}
	d(\mu_t,\nu_t)\le C_Td(\mu_0,\nu_0).
	\end{equation}
\end{proposition}

\vspace{0.5cm}
 
The proof basically follows the same strategy as in \cite{H-Liu}. We first introduce the following notations for simplicity:
	\begin{align*}
	&a(x,\mu_t):=\int_{X}\phi(x_*-x)\frac{v_*}{\theta_*}\mu_t(dz_*),\quad \rho_\phi(x,\mu_t):=\int_{X}\phi(x_*-x)\mu_t(dz_*),\\
	& b(x,\mu_t):=\int_{X}\zeta(x_*-x)\frac{1}{\theta_*}\mu_t(dz_*),\quad\rho_\zeta(x,\mu_t):=\int_{X}\zeta(x_*-x)\mu_t(dz_*).\\
	\end{align*}
	Then, $F[\mu_t]$ and $G[\mu_t]$ in \eqref{FG} can be written in terms of the above functionals:
	\[F[\mu_t](x,v,\theta)=a(x,\mu_t)-\frac{v}{\theta}\rho_\phi(x,\mu_t),\qquad G[\mu_t](x,\theta)=\frac{1}{\theta}\rho_{\zeta}(x,\mu_t)-b(x,\mu_t).\]

We consider a characteristic curve $(x_\mu(t),v_\mu(t),\theta_\mu(t))=(x_\mu(t;0,x,v,\theta),v_\mu(t;0,x,v,\theta),\theta_\mu(t;0,x,v,\theta))$ associated with the measure $\mu$ as a solution to 
\begin{align}
\begin{aligned}\label{eq_char}
&\frac{dx_\mu(t)}{dt}=v_\mu(t),\quad\frac{dv_\mu(t)}{dt}=F[\mu_t](x_\mu(t),v_\mu(t),\theta_\mu(t)),\quad\frac{d\theta_\mu(t)}{dt}=G[\mu_t](x_\mu(t),\theta_\mu(t)),\quad t>0,\\
&(x_\mu(0),v_\mu(0),\theta_\mu(0))=(x,v,\theta)\in \mbox{supp}\mu_0.
\end{aligned}
\end{align}

Contrary to the C-S model, the above forcing terms $F$ and $G$ of the kinetic TCS model are singular at $\theta=0$. Therefore, the main difficulty is to prevent the temperature trajectory $\theta_\mu(t)$ from vanishing in finite time.

 In the following lemma, we provide the positive lower bound of $\theta_\mu(t)$. We also provide $L^\infty$-bound and Lipschitz continuity for the functionals $a,\rho_\phi, b$ and $\rho_\zeta$, and also the stability of $a, b$ with respect to the input measures.

\begin{lemma}\label{lem-all}
	Let $\mu\in L^\infty(0,T;\mathcal{M}(X))$ be a measure-valued solution to \eqref{eq_kinet_TCS} satisfying \eqref{supp-ass}. Then, the following assertions hold:
	\begin{enumerate}
	\item The total mass is conserved:
	\[
	\int_{X}\mu_t(dz)=\int_{X}\mu_0(dz) =:m_0 <\infty ,\quad \forall t\in [0,T] .
	\]
	\item There exists a unique $C^1$-characteristic curve  
	$(x_\mu(t),v_\mu(t),\theta_\mu(t))$ of \eqref{eq_char} on $[0,T]$ such that for some constant $C_T>0$,
	\[ \theta_\mu(t)\ge \theta_m^0,\qquad |v_\mu(t)| \le C_T,  \qquad 0\le t\le T.\]
	\item The $L^\infty$-norm and Lipschitz constants of the functionals $a,\rho_\phi, b$ and $\rho_\zeta$ are bounded:\\
	for all $x, y\in\bbt^d$ and $t\le T$,
	\begin{align*}
	&|a(x,\mu_t)|\le\frac{\|\phi\|_{L^\infty}  P_T m_0}{\theta_m^0},\quad |a(x,\mu_t)-a(y,\mu_t)|\le\frac{\|\phi\|_{\textup{\textup{Lip}}}P_T m_0}{\theta_m^0}|x-y| , \\
	& |\rho_\phi(x,\mu_t)|\le \|\phi\|_{L^\infty} m_0,\quad |\rho_\phi(x,\mu_t)-\rho_\phi(y,\mu_t)|\le \|\phi\|_{\textup{Lip}} m_0 |x-y|, \\
	&|b(x,\mu_t)|\le \frac{\|\zeta\|_{L^\infty} m_0}{\theta_m^0},\quad |b(x,\mu_t)-b(y,\mu_t)|\le \frac{\|\zeta\|_{\textup{Lip}} m_0}{\theta_m^0}|x-y|\\
	&|\rho_\zeta(x,\mu_t)|\le\|\zeta\|_{L^\infty} m_0 ,\quad |\rho_\zeta(x,\mu_t)-\rho_\zeta(y,\mu_t)|\le \|\zeta\|_{\textup{Lip}} m_0 |x-y|.
	\end{align*}
	\item 
	For any $\mu,\nu\in L^\infty(0,T;\mathcal{M}(X))$ satisfying \eqref{supp-ass},
there exists a positive constant $C_T$ such that
for all $x\in\bbt^d$ and $t\le T$,
\begin{align*}
&|a(x,\mu_t)-a(x,\nu_t)|\le C_Td(\mu_t,\nu_t),\quad |b(x,\mu_t)-b(x,\nu_t)|\le C_Td(\mu_t,\nu_t),\\
&|\rho_\phi(x,\mu_t)-\rho_\phi(x,\nu_t)|\le C_Td(\mu_t,\nu_t),\quad |\rho_\zeta(x,\mu_t)-\rho_\zeta(x,\nu_t)|\le C_Td(\mu_t,\nu_t).
\end{align*}
\end{enumerate}
\end{lemma}	

\begin{proof}
(1) We consider $g\equiv 1$ in \eqref{eq_measure} to show
\[
\int_{X}\mu_t(dz)-\int_{X}\mu_0(dz) = \int_0^t  \<\mu_s, 0\> ds =0 .
\] 
\noindent (2) 
Note that for any compact set $D$ in $\bbt^d\times \bbr^d\times\bbr_+$  properly containing $\bbt^d\times B_{P_T}(0)\times[\theta_m^0,\theta_M^0]$, $F[\mu_t](x,v,\theta)$ is uniform Lipschitz continuous in $D$. Moreover, $F[\mu_t](x,v,\theta)$  is continuous in $t\in [0,T]$ by the weak continuity of $t\mapsto \mu_t$. Likewise, $G[\mu_t](x,\theta)$ satisfies the same properties as above. Thus, the Cauchy-Lipschitz theorem implies that the ODE \eqref{eq_char} has a unique $C^1$-characteristic curve $(x_\mu(t),v_\mu(t),\theta_\mu(t))$ up to a local time $T_*$.
Now, we will show that for a maximal existence time $T_M$ of $(x_\mu(t),v_\mu(t),\theta_\mu(t))$, there exists a constant $C(T_M)>0$ such that
\beq\label{claim-cont}
 \theta_\mu(t)\ge \theta_m^0,\qquad |v_\mu(t)| \le C(T_M),  \qquad 0\le t< T_M .
\eeq
Once we prove \eqref{claim-cont}, then the continuation argument implies the global-in-time existence with the desired estimates.
So it remains to prove \eqref{claim-cont}. 
We may first verify the bounds for $\theta_\mu(t)$ by the contradiction argument. Suppose that there exists $t_*\in(0,T_M)$  such that
$0<\theta_\mu (t_*)<\theta_m^0$.
Let $\e:=\theta_m^0-\theta_\mu(t_*)$. Since $\theta_\mu(0)\in[\theta_m^0,\theta_M^0]$ we suppose that without loss of generality, $t_*$ is the first hitting time of $\theta_\mu$ to $\theta_m^0-\e$:
\[t_*:=\inf\{0< t< T_M~:~\theta_\mu(t)\le\theta_m^0-\e\}.\]
Then, $\inf_{0\le t\le t_*} \theta_\mu (t) =\theta_m^0-\e$, which together with the definition of $G[\mu_t]$ yields
\[M:=\sup_{0\le t\le t_*} \left|G[\mu_t](x_\mu(t),\theta_\mu(t)) \right| \le \frac{\|\zeta\|_{L^\infty}m_0(\theta_M^0+\sup_{0\le t\le t_*}\theta_\mu(t))}{(\theta_m^0-\e)\theta_m^0}.\]
Thus, $\theta_\mu(t)$ is Lipschitz continuous on $[0,t_*]$, and $M$ denotes the Lipschitz constant  of $\theta_\mu(t)$ on $0\le t\le t_*$. 
Therefore, $\theta_\mu(t)\in \left(\theta_m^0-\e,\theta_m^0-\frac{\e}{2}\right]$ for all $t\in(t_*-\frac{\e}{2M},t_*)$. \\
Since $\theta_\mu(t_*)<\theta_\mu(t_*-\frac{\e}{2M})$, 
we use the mean-value theorem to find the time $\bar{t}\in \left(t_*-\frac{\e}{2M},t_*\right)$ such that 
\[G[\mu_{\bar{t}}]\left(x_\mu(\bar{t}),\theta_\mu(\bar{t})\right)=\frac{d\theta_\mu(t)}{dt}\Bigg|_{t=\bar{t}}<0.  \]
However, since $\theta_\mu(\bar{t})\in \left(\theta_m^0-\e,\theta_m^0-\frac{\e}{2}\right]$, it follows from \eqref{supp-ass} that
\[G[\mu_{\bar{t}}]\left(x_\mu(\bar{t}),\theta_\mu(\bar{t})\right)=\int_{X}\zeta\left(x_\mu(\bar{t})-x_*\right)\left(\frac{1}{\theta_\mu(\bar{t})}-\frac{1}{\theta_*}\right)\mu_{\bar{t}}(dz_*)\ge0,\]
which yields contradiction. Therefore,  $\theta_\mu(t)\ge \theta_m^0$ for all $0\le t< T_M$.\\

\noindent The second estimate of \eqref{claim-cont} is straightforwardly obtained as follows: 
since for all $t\in(0,T_M)$,
\begin{align*}
\frac{d}{dt}|v_\mu(t)|^2 &= 2 \int_{X} \phi(x_\mu(t)-x_*)\left(\frac{v_*}{\theta_*}\cdot v_\mu(t) - \frac{|v_\mu(t)|^2 }{ \theta_\mu(t)}\right)\mu_{t}(dz_*) \\
&\le 2 |v_\mu(t)| \frac{\|\phi\|_{L^\infty}P_T m_0}{\theta_0^m} \le |v_\mu(t)|^2 + \left(\frac{\|\phi\|_{L^\infty}P_T m_0}{\theta_0^m}\right)^2,
\end{align*}
the Gr\"onwall's lemma gives the bound of $v_\mu(t)$.

\noindent (3) Using the $L^\infty$-bound and Lipschitz continuity of $\phi$, and the boundedness of support of $\mu$, we have 
	\begin{align*}
	&|a(x,\mu_t)|\le \int_{X}\phi(x_*-x)\left|\frac{v_*}{\theta_*}\right|\mu_t(dz_*)\le \frac{\|\phi\|_{L^\infty}  P_T m_0}{\theta_m^0},\\
	&|a(x,\mu_t)-a(y,\mu_t)|\le\int_{X}|\phi(x_*-x)-\phi(x_*-y)|\left|\frac{v_*}{\theta_*}\right|\mu_t(dz_*)\le\frac{\|\phi\|_{\textup{Lip}} P_T m_0}{\theta_m^0}|x-y| .
	\end{align*}
	Likewise, we obtain the remaining estimates. \\

\noindent (4)
Using \eqref{B-1} and the fact that the Lipschitz constant of product of functions are bounded as:
\[\|fgh\|_{\textup{Lip}}\le \|f\|_{\textup{Lip}}\|g\|_{L^\infty}\|h\|_{L^\infty}+\|f\|_{L^\infty}\|g\|_{\textup{Lip}}\|h\|_{L^\infty}+\|f\|_{L^\infty}\|g\|_{L^\infty}\|h\|_{\textup{Lip}},\]
we have
\begin{align}
\begin{aligned}\label{aa}
	|a(x,\mu_s)-a(x,\nu_s)|&=\left|\int_{X}\phi(x_*-x)\frac{v_*}{\theta_*}(\mu_s-\nu_s)(dz_*)\right|\\
	&\le \max\left\{\frac{\|\phi\|_{L^\infty} P_T}{\theta_m^0},\frac{\|\phi\|_{\textup{Lip}}P_T}{\theta_m^0}+\frac{\|\phi\|_{L^\infty}}{\theta_m^0}+\frac{\|\phi\|_{L^\infty} P_T}{(\theta_m^0)^2}\right\}d(\mu_s,\nu_s),
\end{aligned}
\end{align}
where note that although the map $\theta_*\mapsto\frac{1}{\theta_*}$ is not a bounded Lipschitz function on $[0,\infty)$, it is bounded Lipschitz function on $[\theta_m^0,\infty)$, which includes the temperature supports of $\mu_t$ and $\nu_t$. \\
Similarly, we have
	\begin{align*}
	|b(x,\mu_s)-b(x,\nu_s)|&=\left|\int_{X}\zeta(x_*-x)\frac{1}{\theta_*}(\mu_s-\nu_s)(dz_*)\right|\le \max\left\{\frac{\|\zeta\|_{L^\infty}}{\theta_m^0},\frac{\|\zeta\|_{\textup{Lip}}}{\theta_m^0}+\frac{\|\zeta\|_{L^\infty}}{(\theta_m^0)^2}\right\}d(\mu_s,\nu_s).
	\end{align*}
	For $\rho_\phi$ and $\rho_\zeta$, we directly have
	\begin{align*}
	&|\rho_\phi(x,\mu_s)-\rho_\phi(x,\nu_s)|=\left|\int_{X}\phi(x_*-x)(\mu_s-\nu_s)(dz_*)\right|\le \max\{\|\phi\|_{L^\infty},\|\phi\|_{\textup{Lip}}\}d(\mu_s,\nu_s),\\
	&|\rho_\zeta(x,\mu_s)-\rho_\zeta(x,\nu_s)|=\left|\int_{X}\zeta(x_*-x)(\mu_s-\nu_s)(dz_*)\right|\le \max\{\|\zeta\|_{L^\infty},\|\zeta\|_{\textup{Lip}}\}d(\mu_s,\nu_s).
	\end{align*}
\end{proof}

\vspace{0.5cm}

We now use Lemma \ref{lem-all} to estimate the difference between two characteristic curves respectively associated with two measures $\mu$ and $\nu$. To this end, for any fixed $z=(x,v,\theta) \in \bbt^d\times B_{P_T}(0)\times[\theta_m^0,\theta_M^0]$, we denote the differences between the components of the curves by
\[\Delta_x(t):=x_\mu(t;0,z)-x_\nu(t;0,z),\quad \Delta_v(t):=v_\mu(t;0,z)-v_\nu(t;0,z),\quad \Delta_\theta(t):=\theta_\mu(t;0,z)-\theta_\nu(t;0,z),\]
and the total difference by
\[\Delta_z(t):=|\Delta_x(t)|+|\Delta_v(t)|+|\Delta_\theta(t)|.\]

\begin{lemma}\label{lem-final}
		Let $\mu,\nu\in L^\infty([0,T); \mathcal{M}(X))$ be measure-valued solutions to \eqref{eq_kinet_TCS} satisfying \eqref{supp-ass}.
	Let $(x_\mu(t),v_\mu(t),\theta_\mu(t))$ and $(x_\nu(t),v_\nu(t),\theta_\nu(t))$ be the characteristic curves respectively associated with the two measures $\mu$ and $\nu$.   Then, there exists a constant $C_T>0$ such that
	for any $z\in  \bbt^d\times B_{P_T}(0)\times[\theta_m^0,\theta_M^0]$,
	\[\Delta_z(t)\le C_T\int_0^t d(\mu_\tau,\nu_\tau)\,d\tau,\quad 0\le t\le T.\]
\end{lemma}
\begin{proof}
First of all, from the definition, $\Delta_x(t)$ and $\Delta_v(t)$, we have
\[\frac{d\Delta_x(\tau)}{d\tau}=\Delta_v(\tau).\]
Moreover, by (3) and (4) of Lemma \ref{lem-all}, there exists a constant $C_T>0$ such that 
	\begin{align*}
	&|a(x_\mu(t),\mu_t)-a(x_\nu(t),\nu_t)|\le C_T(|x_\mu(t)-x_\nu(t)|+d(\mu_t,\nu_t)),\\
	&|b(x_\mu(t),\mu_t)-b(x_\nu(t),\nu_t)|\le C_T(|x_\mu(t)-x_\nu(t)|+d(\mu_t,\nu_t)),\\
	& |\rho_\phi(x_\mu(t),\mu_t)-\rho_\phi(x_\nu(t),\nu_t)|\le C_T (|x_\mu(t)-x_\nu(t)|+d(\mu_t,\nu_t) ),\\
	& |\rho_\zeta(x_\mu(t),\mu_t)-\rho_\zeta(x_\nu(t),\nu_t)|\le C_T(|x_\mu(t)-x_\nu(t)|+d(\mu_t,\nu_t) ).
	\end{align*}
These estimates together with \eqref{eq_char} and Lemma \ref{lem-all} imply that
\begin{align*}
	\frac{d\Delta_v(\tau)}{d\tau}&=\left(a(x_\mu(\tau),\mu_\tau)-\frac{v_\mu(\tau)}{\theta_\mu(\tau)}\rho_\phi(x_\mu(\tau),\mu_\tau)\right)-\left(a(x_\nu(\tau),\nu_\tau)-\frac{v_\nu(\tau)}{\theta_\nu(\tau)}\rho_\phi(x_\nu(\tau),\nu_\tau)\right)\\
	&\le C_T (|\Delta_x(\tau)|+ d(\mu_\tau,\nu_\tau) ) \\
	&\quad +\left|\frac{v_\mu(\tau)}{\theta_\mu(\tau)}-\frac{v_\nu(\tau)}{\theta_\nu(\tau)}\right|\rho_\phi(x_\mu(\tau),\mu_\tau)+\left|\frac{v_\nu(\tau)}{\theta_\nu(\tau)}\right||\rho_\phi(x_\mu(\tau),\mu_\tau)-\rho_\phi(x_\nu(\tau),\nu_\tau)|\\
	&\le C_T (|\Delta_x(\tau)|+ d(\mu_\tau,\nu_\tau) ) \\
		&\quad +\|\phi\|_{L^\infty}m_0\left(\frac{|\Delta_v(\tau)|}{\theta_m^0}+\frac{C_T|\Delta_\theta(\tau)|}{(\theta_m^0)^2}\right)+C_T\left|\frac{C_T}{\theta_m^0}\right|\left(|\Delta_x(\tau)|+d(\mu_\tau,\nu_\tau)\right)\\
	&\le C_T(|\Delta_x(\tau)|+|\Delta_v(\tau)|+|\Delta_\theta(\tau)|)+C_Td(\mu_\tau,\nu_\tau).
	\end{align*}
	Finally, we estimate $\Delta_\theta(\tau)$ as
	\begin{align*}
	\frac{d\Delta_\theta(\tau)}{d\tau}&=\left(\frac{1}{\theta_\mu(\tau)}\rho_\zeta(x_\mu(\tau),\mu_\tau)-b(x_\mu(\tau),\mu_\tau)\right)-\left(\frac{1}{\theta_\nu(\tau)}\rho_\zeta(x_\nu(\tau),\nu_\tau)-b(x_\nu(\tau),\nu_\tau)\right)\\
	&\le\left|\frac{1}{\theta_\mu(\tau)}-\frac{1}{\theta_\nu(\tau)}\right|\rho_\zeta(x_\mu(\tau),\mu_\tau)+\left|\frac{1}{\theta_\nu(\tau)}\right||\rho_\zeta(x_\mu(\tau),\mu_\tau)-\rho_\zeta(x_\nu(\tau),\nu_\tau)|\\
	&\quad +C_T(|\Delta_x(\tau)|+d(\mu_\tau,\nu_\tau))\\
	&\le \|\zeta\|_{L^\infty}m_0\frac{|\Delta_\theta(\tau)|}{(\theta_m^0)^2}+\frac{C_T}{\theta_m^0}\left(|\Delta_x(\tau)|+d(\mu_\tau,\nu_\tau)\right)+C_T(|\Delta_x(\tau)|+d(\mu_\tau,\nu_\tau))\\
	&\le C_T(|\Delta_x(\tau)|+|\Delta_\theta(\tau)|)+C_Td(\mu_\tau,\nu_\tau).
	\end{align*}
	We now collect the estimates for $\Delta_x$, $\Delta_v$ and $\Delta_\theta$ to obtain
	\[\frac{d\Delta_z(\tau)}{d\tau}\le C_T(\Delta_z(\tau)+d(\mu_\tau,\nu_\tau)),\quad \Delta_z(0)=0.\]
	Therefore, the Gr\"onwall's lemma implies
	\[\Delta_z(t)\le C_T \int_0^t d(\mu_\tau,\nu_\tau)\,d\tau,\quad 0\le t\le T.\]
\end{proof}

\vspace{0.5cm}	

\subsection{Proof of Proposition \ref{prop:st}}	
	 Let $g\in \Omega$ be an arbitrary test function. Then, since $\mu_t$ is a pushforward measure of $\mu_0$ by the map $z_\mu(t;0,z):=(x_\mu, v_\mu, \theta_\mu)(t;0,z)$ (see \cite[Lemma 5.5]{H-Liu}), we have
	\[\int_Xg(z)\mu_t(dz)=\int_X g(z_\mu(t;0,z))\mu_0(dz),\]
	and the exactly same equation holds for $\nu$. Thus, using Lemma \ref{lem-final}, 
	\begin{align*}
	\left|\int_{X}g(z)(\mu_t-\nu_t)(dz)\right|&=\left|\int_{X}g(z_\mu(t;0,z))\mu_0(dz)-\int_{X}g(z_\nu(t;0,z))\nu_0(dz)\right|\\
	&\le\int_{X}|g(z_\mu(t;0,z))-g(z_\nu(t;0,z))|\mu_0(dz)+ \left|\int_{X}g(z_\nu(t;0,z))(\mu_0-\nu_0)(dz)\right|\\
	&\le \int_{X}|z_\mu(t;0,z)-z_\nu(t;0,z)|\mu_0(dz) +d(\mu_0,\nu_0)\\
	&\le \int_X \Delta_z(t)\mu_0(dz)+d(\mu_0,\nu_0)\\
	&\le m_0C_T\int_0^t d(\mu_s,\nu_s)\,ds+d(\mu_0,\nu_0) .
	\end{align*}
	Since $g$ was arbitrary in $\Omega$, we have
	\[d(\mu_t,\nu_t)\le m_0C_T\int_0^t d(\mu_s,\nu_s)\,ds+d(\mu_0,\nu_0),\]
	and the Gr\"onwall's inequality implies the desired estimate.

\section{The proof of  Theorem \ref{T2.2}}\label{sec:4}
Since Proposition \ref{prop:st} provides the stability estimate, and consequently, the uniqueness of measure-valued solutions to \eqref{eq_kinet_TCS}, it remains to show that $f(t,x,v,\theta)=\rho(t,x)\otimes\delta_{u(t,x)}(v)\otimes\delta_{e(t,x)}(\theta)$ is a measure-valued solution to \eqref{eq_kinet_TCS}. We verify whether the left- and right- hand sides of \eqref{eq_measure} are equal.\\

\noindent $\bullet$ (Left-hand side of \eqref{eq_measure}): We substitute $\mu_t=\rho(t,x)\, dx\otimes\delta_{u(t,x)}(v)\otimes\delta_{e(t,x)}(\theta)$ to the left-hand side to get
\[
\mbox{L.H.S}=\int_{\bbt^d} \rho(t,x)g(t,x,u(t,x),e(t,x))\,dx-\int_{\bbt^d}\rho_0(x)g(0,x,u(0,x),e(0,x))\,dx .
\]
Since $(\rho, u,e)\in C^1((0,\infty)\times\bbt^d)$ by Proposition \ref{prop:smooth} together with Sobolev embedding, the following computations make sense: using the continuity equation in \eqref{eq_hydro_TCS},
\begin{align*}
\mbox{L.H.S}&=\int_{\bbt^d}\int_0^t \partial_s(\rho(s,x)g(s,x,u(s,x),e(s,x)))\,ds\,dx\\
&=\int_{\bbt^d}\int_0^t (\partial_s\rho(s,x))(g(s,x,u(s,x),e(s,x)))+\rho(s,x)\partial_s(g(s,x,u(s,x),e(s,x)))\,ds\,dx \\
&=\int_{\bbt^d}\int_0^t \rho u \cdot[(\nabla_xg)(s,x,u,e)+(\nabla_x u) (\nabla_v g)(s,x,u,e)+(\partial_\theta g)(s,x,u,e)\nabla_x e]\,ds\,dx\\
&\quad+\int_{\bbt^d}\int_0^t \rho [(\partial_s g)(s,x,u,e)+(\nabla_vg)(s,x,u,e)\cdot\partial_s u + (\partial_\theta g)(s,x,u,e)\partial_s e]\,ds\,dx\\
&=\int_0^t\int_{\bbt^d}\bigg[(\partial_s g)(s,x,u,e)+u(s,x)\cdot (\nabla_xg)(s,x,u,e) +(\partial_su+u\cdot\nabla_x u)(\nabla_v g)(s,x,u,e)\\
&\hspace{2cm}+(\partial_se+u \cdot \nabla_xe)(\partial_\theta g)(s,x,u,e)\bigg]\rho\,dx\,ds .
\end{align*}
Then, using the equations for momentum and energy in \eqref{eq_hydro_TCS}, we have
\begin{align*}
\mbox{L.H.S}
&=\int_0^t\int_{\bbt^d}\bigg[(\partial_s g)(s,x,u(s,x),e(s,x))+u(s,x)\cdot (\nabla_xg)(s,x,u(s,x),e(s,x))\\
&\hspace{1.5cm}+\left(\int_{\bbt^d}\phi(x-x_*)\left(\frac{u(t,x_*)}{e(t,x_*)}-\frac{u(t,x)}{e(t,x)}\right)\rho(t,x_*)\,dx_*\right)\cdot(\nabla_v g)(s,x,u(s,x),e(s,x))\\
&\hspace{1.5cm}+\left(\int_{\bbt^d}\zeta(x-x_*)\left(\frac{1}{e(t,x)}-\frac{1}{e(t,x_*)}\right)\rho(t,x_*)\,dx_*\right)(\partial_\theta g)(s,x,u(s,x),e(s,x))\bigg]\rho(s,x)\,dx\,ds.
\end{align*}

\noindent $\bullet$ (Right-hand side of \eqref{eq_measure}): Since
\[\mbox{R.H.S}=\int_0^t \int_{\bbt^{d}\times\bbr^d\times\bbr_+}\left[\partial_sg +v\cdot\nabla_x g+F[\rho\delta_u\delta_e]\cdot \nabla_v g+G[\rho\delta_u \delta_e]\right]\delta_{u(s,x)}(dv) \otimes \delta_{e(s,x)}(d\theta) \otimes \rho(s,x)\,dx\,ds,\]
together with 
\[F[\rho  \delta_u\delta_e](x,v,\theta)=\int_{\bbt^d}\phi(x-x_*)\left(\frac{u(t,x_*)}{e(t,x_*)}-\frac{v}{\theta}\right)\rho(t,x_*)\,dx_*,\]
and
\[G[\rho  \delta_u\delta_e](x,\theta)=\int_{\bbt^d}\zeta(x-x_*)\left(\frac{1}{\theta}-\frac{1}{e(t,x_*)}\right)\rho(t,x_*)\,dx_*,\]
we have R.H.S $=$ L.H.S. Therefore, the given mono-kinetic distribution is indeed a measure-valued solution to \eqref{eq_kinet_TCS}. Hence, this and Proposition \ref{prop:st} completes the proof of Theorem \ref{T2.2}.

\bibliographystyle{amsplain}

\end{document}